\documentclass[preprint,11pt]{elsarticle}
\usepackage{amsmath,amscd,amsthm,amssymb,fullpage}
\usepackage{url}
\usepackage{floatflt}

\newtheorem{thm}{Theorem}[section]
\newtheorem{lem}[thm]{Lemma}

\newcommand{\forme}[1]{}

\newcommand{\Or}{{\mathbf O}^{\mathrm \theta}}
\newcommand{\cupdot}{\mathbin{\mathaccent\cdot\cup}}

\begin{document}
\begin{frontmatter}

\title{Isomorphism classes of association schemes induced by Hadamard matrices\tnoteref{fn1}}
\begin{abstract}
Every Hadamard matrix $H$ of order $n > 1$ induces a graph with $4n$ vertices, called the Hadamard graph $\Gamma(H)$ of $H$.
Since $\Gamma(H)$ is a distance-regular graph with diameter $4$, it induces a $4$-class association scheme $(\Omega, S)$ of order $4n$.
In this article we show a way to construct fission schemes of $(\Omega, S)$ under certain conditions, and for such a fission scheme we estimate the
number of isomorphism classes with the same intersection numbers as the fission scheme.
\end{abstract}

\author{Mitsugu Hirasaka}
\ead{hirasaka@pusan.ac.kr}

\author[]{Kijung Kim\corref{cor1}}
\ead{knukkj@pusan.ac.kr}

\author{Hyonju Yu}
\ead{lojs4110@pusan.ac.kr}

\cortext[cor1]{Corresponding author.}
\address{Department of Mathematics, Pusan National University,\\ Busan 609--735, Republic of Korea.}

\tnotetext[fn1]{The first author's research was supported by Basic Science Research Program through the National Research Foundation of Korea(NRF) funded by the Ministry of Education (2013R1A1A2012532). The second author's research was supported by Basic Science Research Program through the National Research Foundation of Korea(NRF) funded by the Ministry of Education (2013R1A1A2005349).}


\end{frontmatter}

\section{Introduction}\label{sec:intro}

According \cite[Theorem 1.8.1]{bcn} we denote by $\Gamma(H)$ the Hadamard graph induced by a Hadamard matrix $H$.
It is known that, for a Hadamard matrix $H$ of order $n > 1$, $\Gamma(H)$ is an antipodal bipartite distance-regular graph with $4n$ vertices and diameter $4$.
Also, \cite[page 58]{bcn} shows that every distance-regular graph of diameter $d$ induces a $d$-class association scheme whose relations are defined by
the distance function of the graph. Thus, $\Gamma(H)$ induces an association scheme $(\Omega, S)$ such that $\Omega$ is the vertex set of $\Gamma(H)$
and $S = \{ s_i \mid 0 \leq i \leq 4 \}$, where $s_i = \{ (x,y) \in \Omega \times \Omega \mid d_{\Gamma(H)}(x,y) = i \}$.
Now we consider a fission scheme of $(\Omega, S)$ such that only $s_2$ is partitioned into $\{ t_1, t_2, \dotsc, t_m  \}$ and
\begin{equation}\label{A}
| xs_4 \cap yt_j | = 1 ~\text{for all}~  j=1,2, \dotsc, m,
\end{equation}
where $(y, x) \in t_j$ and $xs:=\{ y \mid (x, y) \in s \}$ for a binary relation $s$.

Since $s_0 \cup s_2 \cup s_4$ is an equivalence relation on $\Omega$ with exactly two equivalence classes, say $Y_1$ and $Y_2$, the restrictions of
$\{ s_0, s_4, t_1, t_2, \dotsc, t_m \}$ to $Y_i$ form the relation set of an association scheme for $i= 1, 2$.
In this article we deal with association schemes with the same intersection numbers as the fission scheme $(\Omega, \{t_j \mid 1 \leq j \leq m\} \cup \{s_0, s_1, s_3, s_4 \})$.
In the study of association schemes it has been one of the main topics to characterize association schemes, especially related to distance-regular
graphs by intersection numbers (see \cite[Chapter 9]{bcn} or \cite{bannai}), and many $P$- and $Q$- polynomial association schemes with large diameter are
uniquely determined by its intersection numbers: cf. \cite{bcn, ter, ega}. On the other hand, we can find quite many isomorphism classes of association schemes with the same intersection numbers as one can see in Table 3.
But, it does not guarantee that we can find such a huge number of association schemes of order $n$
with the same intersection numbers when $n$ is large enough.
In this article we prove that $(\Omega, \{t_j \mid 1 \leq j \leq m\} \cup \{s_0, s_1, s_3, s_4 \})$ is an association scheme
and give a lower bound for the number of isomorphism classes of association schemes with the same intersection numbers as $(\Omega, \{t_j \mid 1 \leq j \leq m\} \cup \{s_0, s_1, s_3, s_4 \})$.

\vskip10pt
The following are our main results which show the reason why so many isomorphism classes appear as mentioned above.

Let $(X, S)$ be an association scheme of order $n$ and $H$ a Hadamard matrix whose rows and columns are indexed by the elements of $X$.
We denote by $\mathbb{F}_2$ the finite field with two elements.
Also we denote by $x_{ab}$ an element $(x, a, b)$ of $X \times \mathbb{F}_2 \times \mathbb{F}_2$ for short.

Define
\begin{equation}\label{M}
\begin{array}{lll}
\widetilde{X}&=& \{ x_{ab} \mid x \in X, a, b \in \mathbb{F}_2  \}, \\
\widetilde{t}&=&\{(x_{ab}, x_{a(b+1)})\mid x \in X, a, b \in \mathbb{F}_2 \}, \\
\widetilde{s}&=&\{(x_{ab}, y_{ac}) \mid (x, y) \in s, a, b, c \in \mathbb{F}_2  \} \ \  \mbox{for}\ s \in S\setminus\{1_{X}\}, \\
r^{1}_H&= & \{(x_{ab}, y_{cd}) \mid  x, y \in X, a, b, c, d \in \mathbb{F}_2, (1- \delta_{ac})(H^{T(a)})_{xy} = (-1)^{b+d} \}, \\
r^{-1}_H&= & \{(x_{ab}, y_{cd}) \mid x, y \in X, a, b, c, d \in \mathbb{F}_2, a\neq c \} \setminus r^{1}_H, \\
S(H)&=& \{1_{\widetilde{X}}, \widetilde{t}\}\cup \{\widetilde{s}\mid s \in S\setminus\{1_{X}\}\}\cup \{r^{\epsilon}_H\mid\epsilon=\pm 1\},
\end{array}
\end{equation}

where $H^T$ is the transpose of a matrix $H$,

\[ \delta_{ac}   = \left\{
                      \begin{array}{ll}
                      1 & \hbox{if $a=c$;} \\
                      0 & \hbox{otherwise},
                      \end{array}
                     \right.\]

\[H^{T(a)} =  \left\{
                      \begin{array}{ll}
                      H & \hbox{if $a=0$;} \\
                      H^T & \hbox{if $a=1$.}
                      \end{array}
                     \right.\]

Note that $(\widetilde{X}, r^{1}_H)$ is the Hadamard graph $\Gamma(H)$.

\begin{thm}\label{thm:main1}
$(\widetilde{X}, S(H))$ is an association scheme.
\end{thm}
Note that $S(H)$ is a refinement of $C_2 \wr (S \wr C_2)$, where $C_2$ is the relation set of an association scheme of order $2$
(see Section \ref{sec:pre} for the definition of wreath product).

We prepare some notations.
\begin{itemize}
\item For a finite set $X$, $Sym(X)$ is the symmetric group on $X$.
\item For a permutation $\sigma \in Sym(X)$, $P_\sigma$ is the permutation matrix with respect to $\sigma$.
\item For a permutation group $G \leq Sym(X)$, $\mathcal{P}(G) := \{ P_\sigma \mid \sigma \in G \}$.
\item $diag(\varepsilon_x \mid x \in X)$ is a diagonal matrix whose the $(x, x)$-entry is $\varepsilon_x$,
where $\varepsilon: X \rightarrow \mathbb{C}$ is a function defined by $x \mapsto \varepsilon_x$.
\item $D_x$ is a diagonal matrix such that the $(x, x)$-entry is $-1$ and the other diagonal entries are $1$.
\end{itemize}

Let $\mathrm{I} := \mathcal{P}(\mathrm{Iso}(X,S))$ and
$\mathrm{D} := \langle \{ D_x \mid x \in X \} \rangle$ (see Section \ref{sec:pre} for the definition of $\mathrm{Iso}(X,S)$).
Also let $H_1$ and $H_2$ be $n \times n$ Hadamard matrices whose rows and columns are indexed by the elements of $X$.
We say that $H_1$ is \textit{similar} to $H_2$ with respect to $(X,S)$ if there exist $(P',P), (Q',Q) \in \mathrm{D} \times \mathrm{I}$ such that $H_2 = (P'P)^{-1} H_1 Q'Q$ or $H_2^T = (P'P)^{-1} H_1 Q'Q$,
and $(P'P)^{-1} Q'Q \in \mathrm{D}\rtimes \mathcal{P}(\mathrm{Aut}(X, S))$
(see Section \ref{sec:pre} for the definition of $\mathrm{Aut}(X,S)$).
Note that the similarity is an equivalence relation.

\begin{thm}\label{thm:main2}
$(\widetilde{X}, S(H_1))$ is isomorphic to $(\widetilde{X}, S(H_2))$ if and only if
$H_1$ is similar to $H_2$ with respect to $(X,S)$.
\end{thm}

Recall that two Hadamard matrices are \textit{equivalent} if one can be transformed into the other by a series of row or column permutations
or negations.
We define a group
\[\mathrm{Aut}_{x_0}(H) = \{ (\sigma, \tau) \in Sym(X) \times Sym(X) \mid P_\sigma^{-1} H P_\tau = H, \sigma(x_0)=x_0, \tau(x_0)=x_0 \}. \]

\begin{thm}\label{thm:main3}
Let $(X, S)$ be an association scheme of order $n$, $H_0$ a Hadamard matrix whose rows and columns are indexed by the elements of $X$, and $x_0 \in X$.
Then there are at least \[\frac{(n-1)! (n-1)!}{ 2 |\mathrm{Aut}_{x_0}(H_0)| |\mathrm{Aut}(X,S)| |\mathrm{Iso}(X,S)|}\]
isomorphism classes
of association schemes in $\{ (\widetilde{X}, S(H)) \mid H ~\text{is equivalent to}~ H_0 \}$.
\end{thm}

This article is organized as follows.
In Section \ref{sec:pre}, we prepare some terminology and notations.
In Section \ref{sec:main}, we give proofs of the main theorems.
In Section \ref{sec:table}, we list tables related to the main results when $n = 4, 8$.

\section{Preliminaries}\label{sec:pre}

Based on \cite{zies,zies2}, we use the notation on association schemes.
Let $X$ be a non-empty finite set.
Let $S$ denote a partition of $X \times X$. Then the pair $(X, S)$ is
an \textit{association scheme} (or shortly  \textit{scheme}) if it satisfies the following conditions:

\begin{enumerate}
\item $1_{X} := \{ (\alpha, \alpha) \mid  \alpha \in X \} \in S$;
\item For each $s \in S$, $s^* := \{(\alpha, \beta) \mid (\beta, \alpha) \in s \} \in S$;
\item For all $ s, t, u \in S$, $c_{st}^u := | \alpha s \cap \beta t^* |$ is constant whenever $(\alpha, \beta) \in u$,
\end{enumerate}
where $\alpha s:=\{ \beta \in X \mid (\alpha, \beta) \in s \}$.

The numbers $\{ c_{st}^u \mid s, t, u \in S \}$ are called the \textit{intersection numbers} of $S$.
For each $ s \in S$, we abbreviate $c_{ss^*}^{1_X} $ as $n_s$, which is called the \textit{valency} of $s$.
We call $\sum_{s \in S} n_s$ the \textit{order} of $(X, S)$ which is equal to $|X|$.

The \textit{thin residue} of $S$ is the smallest subset $\Or(S)$ of $S$
such that $\bigcup_{t\in \Or(S)}t$ is an equivalence on $X$ and the factor scheme of $(X,S)$ over $\Or(S)$
is induced by a regular permutation group (see \cite[page 45]{zies2} for the definitions).

For each $s \in S$, we denote by $A_s$ the \textit{adjacency matrix} of $s$.
Namely $A_s$ is a matrix whose rows and columns are indexed by the elements of $X$ and $(A_s)_{xy} = 1$ if $(x,y) \in s$ and $(A_s)_{xy} = 0$ otherwise.

\vskip5pt
Let $(X,S)$ and  $(X_1, S_1)$ be association schemes. A bijective mapping $\phi : X \cup S \rightarrow X_1 \cup S_1$ is called an
\textit{isomorphism} from $(X,S)$ to $(X_1, S_1)$ if it satisfies the following conditions:
\begin{enumerate}
\item $\phi(X) \subseteq X_1$ and $\phi(S) \subseteq S_1$;
\item For all $x, y \in X$ and $s \in S$ with $(x,y) \in s$, $(\phi(x), \phi(y)) \in \phi(s)$.
\end{enumerate}
We denote by $\mathrm{Iso}(X,S)$ the set of isomorphisms from $(X,S)$ to itself.
An isomorphism $\phi : X \cup S \rightarrow X \cup S$ is called an \textit{automorphism} of $(X, S)$ if $\phi(s) = s$ for all $s \in S$.
We denote by $\mathrm{Aut}(X,S)$ the automorphism group of $(X, S)$.

We say that two association schemes $(X,S)$ and  $(X_1, S_1)$ are \textit{algebraically isomorphic} or have the \textit{same intersection numbers} if there exists a bijection $\iota : S \rightarrow S_1$ such that $c_{rs}^{t} = c_{\iota(r) \iota(s)} ^{\iota(t)}$ for all $r, s, t \in S$.

\vskip5pt
Let $(W,F)$ and $(Y,H)$ be association schemes.
For each $f \in F$ we define
\[ \overline{f}:= \{ ((w_1,y),(w_2,y)) \mid y \in Y, (w_1,w_2) \in f  \} .\]
For each $h \in H \setminus \{1_Y\}$ we define
\[ \overline{h}:=\{ ((w_1,y_1),(w_2,y_2)) \mid w_1, w_2 \in W, (y_1,y_2) \in h \}. \]
Denote $F \wr H := \{ \overline{f} \mid f \in F \} \cup \{\overline{h} \mid h \in H \setminus \{1_Y\} \}$.
Then $(W \times Y, F \wr H)$ is an association scheme called the \textit{wreath product} of $(W,F)$ and $(Y,H)$.

\section{Proofs of the main theorems}\label{sec:main}

\begin{flushleft}
\textbf{Proof of Theorem~\ref{thm:main1}}
\end{flushleft}
Remark that  $A_{r_H^1}$ forms the adjacency matrix of the Hadamard graph $\Gamma(H)$.
Then $A_{1_{\widetilde{X}}}$, $A_{r_H^1}$, $A_{S^\sqcup}$, $A_{r_{H}^{-1}}$, and $A_{\tilde{t}}$ are the adjacency matrices of distance $i$ graphs $\Gamma(H)_i \ (0 \leq i \leq 4)$, where $S^\sqcup=\bigcup_{s \in S \setminus \{1_X\}}\widetilde{s}$.
Since $\Gamma(H)$ is distance-regular, $(\widetilde{X}, \{1_{\widetilde{X}}, r_H^1, S^\sqcup, r_{H}^{-1}, \tilde{t}\})$ forms an association scheme.
So, it suffices to show that $A_{r_1} A_{r_2}$ is a linear combination of $\{A_{1_{\widetilde{X}}}, A_{r_H^1},A_{r_{H}^{-1}}, A_{\tilde{t}} \}\cup \{A_{\widetilde{s}}\mid s\in S\setminus \{1_X\}\}$, where at least one of $r_1$ and $r_2$ is in $\{\widetilde{s}\mid s\in S\setminus \{1_X\}\}$. From now on, let $s_0 \in S\setminus \{1_X\}\}$.

By the definition of $S(H)$, $$A_{\widetilde{s_0}}A_{1_{\widetilde{X}}}=A_{1_{\widetilde{X}}}A_{\widetilde{s_0}}=A_{\widetilde{s_0}}$$ and  $$A_{\widetilde{s_0}}A_{\widetilde{t}}=A_{\widetilde{t}}A_{\widetilde{s_0}}=A_{\widetilde{s_0}}.$$

Since $(X,S)$ is an association scheme, for $s_1 \in S\setminus\{1_X\}$, $A_{s_0}A_{s_1}=\sum_{s\in S}c^s_{s_0 s_1}A_s$ and $A_{s_1}A_{s_0}=\sum_{s\in S}c^s_{s_1 s_0}A_s$, where $c_{s_0 s_1}^s$'s and $c_{s_1 s_0}^s$'s are intersection numbers of $(X,S)$.
For $x_{ab}, y_{cd}\in \widetilde{X}$, let us consider the set $$Z:=\{z_{ef}\in \widetilde{X}\mid  (x_{ab}, z_{ef}) \in \widetilde{s_0} \ \mbox{and} \  (z_{ef}, y_{cd}) \in \widetilde{s_1}\}.$$
Then for $z_{ef}\in Z$, $a=e$ and $(x,z)\in s_0$ as $(x_{ab}, z_{ef}) \in \widetilde{s_0}$, and $e=c$ and $(z,y)\in s_1$ as $(z_{ef}, y_{cd}) \in \widetilde{s_1}$. And we can check the following facts:
\begin{enumerate}
\item If $(x_{ab}, y_{cd})\in 1_{\widetilde{X}}\cup \widetilde{t}$ then $x=y$ and $a=c$. So, there are $c_{s_0 s_1}^{1_X}$ choices of $z$, one choice of $e$ and two choices of $f$. Thus, $| Z | =2\cdot c_{s_0 s_1}^{1_X}$;
\item If $(x_{ab},y_{cd})\in \widetilde{s}$ for some $s\in S\setminus\{1_X\}$, then $(x,y)\in s$ and $a=c$. So, there are $c_{s_0 s_1}^{s}$ choices of $z$, one choice of $e$ and two choices of $f$. Thus, $| Z | =2\cdot c_{s_0 s_1}^{s}$;
\item If $(x_{ab}, y_{cd})\in r_H^1\cup r_H^{-1}$, then $a\neq c$. Since $a=e=c$ for any $z_{ef}\in Z$, $| Z |=0$.
\end{enumerate}
Combining (i), (ii) and (iii), we obtain $$A_{\widetilde{s_0}}A_{\widetilde{s_1}}=\sum_{s\in S\setminus \{1_X\}}2 \cdot c^{s}_{s_0 s_1}A_{\widetilde{s}}+2\cdot c_{s_0 s_1}^{1_X}(A_{1_{\widetilde{X}}}+A_{\widetilde{t}}).$$
And by symmetric argument, we can show that $$A_{\widetilde{s_1}}A_{\widetilde{s_0}}=\sum_{s\in S\setminus\{1_X\}}2 \cdot c^{s}_{s_1 s_0}A_{\widetilde{s}}+2\cdot c_{s_1 s_0}^{1_X}(A_{1_{\widetilde{X}}}+A_{\widetilde{t}}).$$

Now again, for $x_{ab}, y_{cd}\in \widetilde{X}$, let us consider the set $$Z':=\{ z_{ef}\in \widetilde{X}\mid (x_{ab}, z_{ef}) \in \widetilde{s_0}  \ \mbox{and} \  (z_{ef}, y_{cd}) \in r_{H}^1\}.$$
Then $a=e$ and $(x,z)\in s_0$ as $(x_{ab}, z_{ef}) \in \widetilde{s_0}$, and $e\neq c$ and $(H^{T(e)})_{zy}=(-1)^{f+d}$ as $(z_{ef}, y_{cd}) \in r_{H}^1$. And we can check the following facts:
\begin{enumerate}
\item If $(x_{ab}, y_{cd})\in 1_{\widetilde{X}}\cup \widetilde{t}\cup S^\sqcup$, then $a=c$. Since $a=e\neq c$ for any $z_{ef}\in Z'$, $|Z'| = 0$;
\item If $(x_{ab}, y_{cd})\in r_H^1\cup r_H^{-1}$, then $a\neq c$ and so, there is one choice of $e$. Since $(x,z)\in s_0$, there are $n_{s_0}$ choices of $z$. And for fixed $d, e, y$ and $z$, there is one choice of $f$ as $(H^{T(e)})_{zy}=(-1)^{f+d}$. Thus, $|Z'| = n_{s_0}$.
\end{enumerate}
Combining (i) and (ii), we obtain $$A_{\widetilde{s_0}}A_{r_H^{1}}=n_{s_0}(A_{r_H^1} +A_{r_H^{-1}}).$$
By symmetric argument, we can show that $$A_{r_H^1}A_{\widetilde{s_0}}=n_{s_0}(A_{r_H^1} +A_{r_H^{-1}}).$$

Since $ r^{-1}_H = \{(x_{ab}, y_{cd}) \mid (1- \delta_{ac})(H^{T(a)})_{xy} = (-1)^{b+d+1}, x, y \in X, a, b, c, d \in \mathbb{F}_2 \}$,
a similar argument yields $$A_{\widetilde{s_0}}A_{r_H^{-1}}=A_{r_H^{-1}}A_{\widetilde{s_0}}=n_{s_0}(A_{r_H^1} +A_{r_H^{-1}}).$$
This completes the  proof of Theorem \ref{thm:main1}.
\qed

\begin{lem}\label{lem:auto}
For a Hadamard matrix $H$ and an association scheme $(X,S)$,
let $(\widetilde{X}, S(H))$ be the association scheme constructed by $(\ref{M})$.
Then we have the following:
\begin{enumerate}
\item $\phi : x_{ab} \mapsto x_{(a+1)b}, h \mapsto \phi(h)$ is an isomorphism from $\widetilde{X} \cup S(H)$ to $\widetilde{X} \cup S(H^T)$;
\item For $y \in X$, the transposition $\alpha_y = (y_{ab}~ y_{a(b+a+1)}) \in Sym(\widetilde{X})$ induces an isomorphism $\phi$
from $\widetilde{X} \cup S(H)$ to $\widetilde{X} \cup S(H_1)$ defined by $\phi|_{\widetilde{X}}=\alpha_y , \phi|_{S(H)} : h \mapsto \phi(h)$;
\item For $y \in X$, the transposition $\beta_y = (y_{ab}~ y_{a(b+a)}) \in Sym(\widetilde{X})$ induces an isomorphism
from $\widetilde{X} \cup S(H)$ to $\widetilde{X} \cup S(H_2)$ defined by $\phi|_{\widetilde{X}}=\beta_y , \phi|_{S(H)} : h \mapsto \phi(h)$;
\item $\phi : x_{ab} \mapsto x_{a(b+a+1)}, h \mapsto \phi(h)$ is an isomorphism from $\widetilde{X} \cup S(H)$ to $\widetilde{X} \cup S(-H)$;
\item $\phi : x_{ab} \mapsto x_{a(b+a)}, h \mapsto \phi(h)$ is an isomorphism from $\widetilde{X} \cup S(H)$ to $\widetilde{X} \cup S(-H)$,
\end{enumerate}
where $\phi(h):= \{ (\phi(x), \phi(y)) \mid (x, y) \in h \}$ for $h \in S(H)$, $H_1 = D_y H$ and $H_2 = H D_y$.
\end{lem}
\begin{proof}
In the cases (i), (ii) and (iii), it is easy to see that $\phi$ is bijective on $\widetilde{X}$, $\phi(1_{\widetilde{X}}) = 1_{\widetilde{X}}$, $\phi(\widetilde{t}) = \widetilde{t}$ and $\phi(\widetilde{s}) = \widetilde{s}$ for $s \in S \setminus \{ 1_X \}$.
\begin{enumerate}
\item
Let $(x_{ab}, y_{cd}) \in r^{1}_H$.
Then $(\phi(x_{ab}), \phi(y_{cd})) = (x_{(a+1)b}, y_{(c+1)d})$.
Since $H^{T(a)} = (H^T)^{T(a+1)}$, we have $(\phi(x_{ab}), \phi(y_{cd})) \in r^{1}_{H^T}$.
This means that $\phi(r^{1}_H) = r^{1}_{H^T}$.

\item
Let $(y_{ab}, z_{cd}) \in r^{1}_H$.
When $a=0$, we have $(\alpha_y(y_{ab}), \alpha_y(z_{cd})) = (y_{0(b+1)}, z_{cd})$ and
\[ (y_{0(b+1)}, z_{cd}) \in r^{1}_{H_1} \Leftrightarrow (1 - \delta_{0c})((H_1)^{T(0)})_{yz} = (-1)^{b+d+1}.\]

Since $(1 - \delta_{0c})((H_1)^{T(0)})_{yz} = (-1)(1 - \delta_{0c})(H^{T(0)})_{yz}  = (-1)(-1)^{b+d}$,
we have $(\alpha_y(y_{ab}), \alpha_y(z_{cd})) \in r^{1}_{H_1}$.
This implies that $\phi(r^{1}_H) = r^{1}_{H_1}$.

\item
Let $(y_{ab}, z_{cd}) \in r^{1}_H$.
When $a=1$, we have $(\beta_y(y_{ab}), \beta_y(z_{cd})) = (y_{1(b+1)}, z_{cd})$
Similarly to (ii), we can show $(\beta_y(y_{ab}), \beta_y(z_{cd})) \in r^{1}_{H_2}$.
This means that $\phi(r^{1}_H) = r^{1}_{H_2}$.

\item
Since $\phi$ represents $\Pi_{x \in X} \alpha_x$ as a permutation,
(ii) implies that $\phi : x_{ab} \mapsto x_{a(b+a+1)}$ is an isomorphism $\widetilde{X} \cup S(H) \rightarrow \widetilde{X} \cup S(-H)$.

\item
Since $\phi$ represents $\Pi_{x \in X} \beta_x$ as a permutation,
(iii) implies that $\phi : x_{ab} \mapsto x_{a(b+a)}$ is an isomorphism $\widetilde{X} \cup S(H) \rightarrow \widetilde{X} \cup S(-H)$.
\end{enumerate}
\end{proof}


\begin{flushleft}
\textbf{Proof of Theorem~\ref{thm:main2}}
\end{flushleft}
For convenience, we denote $\{ x_{0b} \mid b \in \mathbb{F}_2, x \in X \}$ by $X_0$
and $\{ x_{1b} \mid b \in \mathbb{F}_2, x \in X \}$ by $X_1$.
\vskip10pt

$(\Rightarrow):$
Assume $(\widetilde{X}, S(H_1)) \simeq (\widetilde{X}, S(H_2))$.
We denote by $\phi$ an isomorphism from $(\widetilde{X}, S(H_1))$ to $(\widetilde{X}, S(H_2))$.
Since $\Or(S(H_1)) = \{1_{\widetilde{X}}, \widetilde{t}\}\cup \{\widetilde{s}\mid s \in S\setminus\{1_{X}\}\}$
and $\phi(\Or(S(H_1))) = \Or(S(H_2))$, $\phi (X_0)$ is either $X_0$ or $X_1$.

By Lemma \ref{lem:auto}(i), there exists an isomorphism from $(\widetilde{X}, S(H_2))$ to $(\widetilde{X}, S(H_2^T))$.
So, we may assume $\phi (X_0) = X_0$.

Since $\phi(\Or(S(H_1))) = \Or(S(H_2))$ and $r_{H_1}^1 \in S(H_1) \setminus \Or(S(H_1))$, $\phi (r_{H_1}^1)$ is either $r_{H_2}^1$ or $r_{H_2}^{-1}$.
By Lemma \ref{lem:auto}(iv), there exists an isomorphism from $(\widetilde{X}, S(H_2))$ to $(\widetilde{X}, S(-H_2))$.
So, we also assume $\phi (r_{H_1}^1) = r_{H_2}^1 ~\text{and}~ \phi (r_{H_1}^{-1}) = r_{H_2}^{-1}$.

Using the fact that $\phi(x_{ab}) = y_{ad}$ for some $y_{ad} \in \widetilde{X}$,
we define $\sigma_a \in Sym(X)$ and $\tau_a : X \rightarrow \mathbb{F}_2$ such that
$\phi(x_{ab}) = \sigma_a(x)_{a (b + \tau_a(x))}$ for each $x_{ab} \in \widetilde{X}$.

Now we check that $\tau_a(x)$ is well defined.
It suffices to show that $\tau_a(x)$ does not depend on $b$ of $x_{ab}$.
If $\phi(x_{ab}) = y_{ad}$ and $\phi(x_{ab'}) = y_{ad'}$,
then $y_{ad} = \sigma_a(x)_{a (b + \tau_a(x))}$ and $y_{ad'} = \sigma_a(x)_{a (b' + \tau_a(x))}$.
Since $\phi$ is well defined, we have $d \neq d'$ for $b \neq b'$.
This means that $\tau_a(x) = b + d = b' + d'$.

We consider two matrices as follows.
\[ P_a := P_{\sigma_a}  ~\text{and}~  Q_a := diag((-1)^{\tau_a(x)} \mid x \in X). \]
\vskip10pt

First, we claim that $(Q_0 P_0)^{-1} H_1 Q_1 P_1$ is similar to $H_1$.
For each $a \in \mathbb{F}_2$, clearly $(Q_a, P_a) \in \mathrm{D} \times \mathrm{I}$.
In order to show $(Q_0 P_0)^{-1} Q_1 P_1 \in \mathrm{D}\rtimes \mathcal{P}(\mathrm{Aut}(X, S))$,
it suffices to check $P_0^{-1} P_1 \in \mathcal{P}(\mathrm{Aut}(X, S))$, since $\mathrm{D}$ is normal in $\mathrm{D}\rtimes \mathcal{P}(\mathrm{Aut}(X, S))$.
Let $(x_{0b}, y_{0c}), (x_{1b}, y_{1c}) \in \widetilde{s}$.
Then $(\sigma_0(x)_{0(b+\tau_0(x))}, \sigma_0(y)_{0(b+\tau_0(y))})$, $(\sigma_1(x)_{1(b+\tau_1(x))},
\sigma_1(y)_{1(b+\tau_1(y))})$ $\in \phi(\widetilde{s}) = \widetilde{s_1}$ for some $s_1 \in S$.
By the definition of $s_1$, we have $(\sigma_0(x), \sigma_0(y)), (\sigma_1(x), \sigma_1(y)) \in s_1$.
This implies $P_0^{-1} P_1 \in \mathcal{P}(\mathrm{Aut}(X, S))$.

\vskip10pt
Next, we claim that $(\widetilde{X}, \phi(S(H_1))) = (\widetilde{X}, S((Q_0 P_0)^{-1} H_1 Q_1 P_1))$, i.e., $H_2 = (Q_0 P_0)^{-1} H_1 Q_1 P_1$.
It suffices to verify that $(x', y')$-entry of $H_2$ is equal to that of $(Q_0 P_0)^{-1} H_1 Q_1 P_1$.

Let $(x_{ab}, y_{cd}) \in r_{H_1}^1$.
Then $(x'_{ab'}, y'_{cd'}) \in r_{H_2}^1$, where $\phi(x_{ab})= x'_{ab'}$ and $\phi(y_{cd})= y'_{cd'}$.

If $a=0$, then $(H_1)_{xy} = (-1)^{b + d}$ and $(H_2)_{x'y'} = (-1)^{b' + d'}$.
Since
\[(P_0^{-1} Q_0^{-1} H_1 Q_1 P_1)_{x'y'} = (Q_0^{-1} H_1 Q_1)_{xy} = (-1)^{\tau_0(x)} (H_1)_{xy} (-1)^{\tau_1(y)},\]
\[b + \tau_0(x) = b' ~~\text{and}~~ d + \tau_1(y) = d',\]

we have
\begin{eqnarray*}
  (P_0^{-1} Q_0^{-1} H_1 Q_1 P_1)_{x'y'} &=& (Q_0^{-1} H_1 Q_1)_{xy}  \\
                                         &=& (-1)^{\tau_0(x)} (H_1)_{xy} (-1)^{\tau_1(y)} \\
                                         &=& (-1)^{b' + d'} = (H_2)_{x'y'}.
\end{eqnarray*}

If $a=1$, then $(H_1^{T})_{xy} = (-1)^{b + d}$ and $(H_2^{T})_{x'y'} = (-1)^{b' + d'}$.
Since
\[(P_1^{T} Q_1^{T}  H_1^{T} (Q_0^{-1})^{T} (P_0^{-1})^{T} )_{x'y'} = (Q_1^{T} H_1^{T} (Q_0^{-1})^{T})_{xy} = (-1)^{\tau_0(y)} (H_1)_{yx} (-1)^{\tau_1(x)},\]
\[b + \tau_1(x) = b' ~~\text{and}~~ d + \tau_0(y) = d',\]

we have
\begin{eqnarray*}
  (P_1^{T} Q_1^{T}  H_1^{T} (Q_0^{-1})^{T} (P_0^{-1})^{T} )_{x'y'} &=& (P_1^{-1} Q_1^{T}  H_1^{T} (Q_0^{-1})^{T} P_0)_{x'y'} \\
                                           &=& (Q_1^{T} H_1^{T} (Q_0^{-1})^{T})_{xy}  \\
                                           &=& (-1)^{\tau_0(y)} (H_1)_{yx} (-1)^{\tau_1(x)} \\
                                           &=& (-1)^{b' + d'} = (H_2^{T})_{x'y'}.
\end{eqnarray*}

\vskip10pt
$(\Leftarrow):$
Assume that $H_1$ is similar to $H_2$ with respect to $(X,S)$.
Since $(\widetilde{X}, S(H_2))$ is isomorphic to $(\widetilde{X}, S(H_2^T))$,
it suffices to consider the case $H_2 = P'P H_1 QQ'$, where $P, Q \in \mathrm{I}$ and $P', Q' \in \mathrm{D}$.
Then $P'P$ can be decomposed into $D_{x_1} D_{x_2} \cdots D_{x_m} P_{\sigma_0}$, where $P = P_{\sigma_0}$ and $D_{x_i} \in \mathrm{D}$ $(i=1, \dotsc, m)$.

According to the fact that $D_{x_i}$ has $-1$ at the entry corresponding to $x_i \in X$,
we define a transposition $\phi_i := ((x_i)_{0b} (x_i)_{0(b+1)}) \in Sym(X_0)$ $(i=2, \dotsc, m)$.
We define a permutation $\phi_{\sigma_0}$ in $Sym(X_0)$ by $\phi_{\sigma_0} (x_{0b}) = (\sigma_0(x))_{0b}$.
Put $\phi := \phi_1 \cdots \phi_m \phi_{\sigma_0}$.

\vskip10pt
Also, $QQ'$ can be decomposed into $Q_{\sigma_1} D_{y_1} \cdots D_{y_l}$, where $Q = Q_{\sigma_1}$ and $D_{y_i} \in \mathrm{D}$ $(i=1, \dotsc, l)$.
Similarly, we define $\psi_i$ $(i=1, \dotsc, l)$ and $\psi_{\sigma_1}$ in $Sym(X_1)$ as the above $\phi_i$ and $\phi_{\sigma_0}$.
Put $\psi :=  \psi_1 \cdots \psi_l \psi_{\sigma_1}$.

\vskip10pt
We claim that $\phi \cup \psi$ is an isomorphism from $(\widetilde{X}, S(H_1))$ to $(\widetilde{X}, S(H_2))$.
It is easy to see that $\phi \cup \psi$ is bijective on $\widetilde{X}$, $(\phi \cup \psi)(1_{\widetilde{X}}) = 1_{\widetilde{X}}$ and
$(\phi \cup \psi)(\tilde{t}) = \tilde{t}$.
Since $\sigma_0 \sigma_1^{-1} \in \mathrm{Aut}(X, S)$,
for each $\tilde{s} \in S(H_1) \setminus \{ 1_{\widetilde{X}}, \tilde{t}, r_{H_1}^\epsilon \}$,
we have $(\phi \cup \psi)(\tilde{s}) = (\phi_1 \cdots \phi_m \cup \psi_1 \cdots \psi_l)(\tilde{s'})$ for some $\tilde{s'} \in S(H_1)$.
Since $\phi_i$ and $\psi_j$ $(1 \leq i \leq m, 1 \leq j \leq l)$ preserve $\tilde{s'}$, we have $(\phi \cup \psi)(\tilde{s}) \in S(H_2)$.

We can check the following facts:
\begin{enumerate}
\item $\phi_i$ $(i = 1, \dotsc, m)$ and $\phi_{\sigma_0}$ correspond to multiplying $-1$ to only one row of $H_1$ and permuting rows of $H_1$, respectively;
\item $\psi_i$ $(i = 1, \dotsc, l)$ and $\psi_{\sigma_1}$ correspond to multiplying $-1$ to only one column of $H_1$ and permuting columns of $H_1$,
respectively.
\end{enumerate}
This implies $\phi (r_{H_1}^\epsilon) = r_{H_2}^\epsilon$.
Therefore, $\phi \cup \psi$ is an isomorphism from $(\widetilde{X}, S(H_1))$ to $(\widetilde{X}, S(H_2))$.
\qed

\begin{flushleft}
\textbf{Proof of Theorem~\ref{thm:main3}}
\end{flushleft}

We define an action of $GL_n(\mathbb{C}) \times GL_n(\mathbb{C})$ on $Mat_X(\mathbb{C})$ by $(H, (P, Q)) \mapsto P^{-1}HQ$.
Restricting our attention to the following subgroups of $GL_n(\mathbb{C}) \times GL_n(\mathbb{C})$ and a subset of $Mat_X(\mathbb{C})$,
we observe their orbits.

Let
\begin{equation}\label{B}
G = \{ (P, Q) \mid P, Q \in \mathrm{D}\rtimes \mathcal{P}(Sym(X)) \},
\end{equation}
\begin{equation}\label{C}
K = \{ (P, Q) \mid P, Q \in \mathrm{D}\rtimes \mathrm{I}, PQ^{-1} \in \mathrm{D}\rtimes \mathcal{P}(\mathrm{Aut}(X,S)) \},
\end{equation}
\begin{equation}\label{D}
G_{x_0} = \{ (P, Q) \mid P, Q \in \mathrm{D}\rtimes \mathcal{P}(Sym(X)_{x_0}),  \}
\end{equation}
where $Sym(X)_{x_0}  = \{ \sigma \in Sym(X) \mid \sigma(x_0) = x_0  \}$.

Let $\mathcal{H}$ be the set of Hadamard matrices of order $n$.
We consider an orbit of $G$ acting on $\mathcal{H}$.
Then the orbit $H_0^G$ is the set of Hadamard matrices which are equivalent to $H_0$, and
decomposed into
\[ \bigcup_{i=1}^r H_i^K,\]
where $H_i \in H_0^G$.
In particular, each orbit of $K$-action on $H_0^G$ is a subset of a similarity class with respect to $(X,S)$,
since for each $1 \leq i \leq r$, the transpose of $H_i$ may not belong to $H_i^{K}$.

By Theorem \ref{thm:main2}, the number of isomorphism classes of association schemes in $\{ (\widetilde{X}, S(H)) \mid H ~\text{is equivalent to}~ H_0 \}$ is at least $\frac{r}{2}$.

Now we give a lower bound for $r$.
Since each orbit $H_i^K$ contains a normalized Hadamard matrix,
without loss of generality, we may assume that $H_0, H_1, \dotsc, H_r$ are normalized Hadamard matrices.
By the orbit-stabilizer property (see \cite[page 57]{rotman}), we have
\[\frac{|G_{x_0}|}{|(G_{x_0})_{H_0}|} = |H_0^{G_{x_0}}| \leq |H_0^G|,\]
where $(G_{x_0})_{H_0}$ is the stabilizer of $H_0$.

\vskip10pt
\textbf{Claim 1}: \[\frac{|G_{x_0}|}{|(G_{x_0})_{H_0}|} = \frac{(n-1)! (n-1)!}{|\mathrm{Aut}_{x_0}(H_0)|} \frac{|\mathrm{D}|^2}{2}.\]

First of all, we verify $(G_{x_0})_{H_0} = \{ \pm(I_n, I_n) \} \mathcal{P}(\mathrm{Aut}_{x_0}(H_0))$.
Every element of $G_{x_0}$ is decomposed into $(P_1P_2, Q_1Q_2)$, where $P_1, Q_1 \in \mathcal{P}(Sym(X)_{x_0})$ and $P_2, Q_2 \in \mathrm{D}$.
We consider all $(P_1P_2, Q_1Q_2)$ such that
\begin{equation}\label{E}
P_2^{-1}P_1^{-1}H_0Q_1Q_2 = H_0.
\end{equation}

Since $H_0$ is normalized, $P_1^{-1}H_0Q_1$ is also normalized.
So, (\ref{E}) implies that $(P_2, Q_2)$ is either $(I_n, I_n)$ or $(-I_n, -I_n)$.
Whichever the case may be, $(P_1,Q_1)$ must satisfy $P_1^{-1}H_0Q_1 = H_0$.
This implies $(P_1, Q_1) \in \mathcal{P}(\mathrm{Aut}_{x_0}(H_0))$.
Thus, $|(G_{x_0})_{H_0}| = 2 \cdot |\mathrm{Aut}_{x_0}(H_0)|$. This completes the proof of Claim 1.
\vskip15pt

\textbf{Claim 2}: \[|K| = |\mathrm{Iso}(X, S)| \cdot |\mathrm{Aut}(X, S)| \cdot |\mathrm{D}|^2.\]

Since $\mathrm{Aut}(X, S)$ is a normal subgroup of $\mathrm{Iso}(X, S)$: cf. \cite[page 3]{mp1},
it is easy to see that $\mathrm{D}\rtimes \mathcal{P}(\mathrm{Aut}(X,S))$ is a normal subgroup of $\mathrm{D}\rtimes \mathrm{I}$.
The following is left as an exercise for the reader.
The group $K$ is a subgroup of $(\mathrm{D}\rtimes \mathrm{I}) \times (\mathrm{D}\rtimes \mathrm{I})$ and
its order is $|\mathrm{D}\rtimes \mathcal{P}(Sym(X))| \cdot |\mathrm{D}\rtimes \mathcal{P}(\mathrm{Aut}(X,S))|$.
\vskip15pt

Applying the orbit-stabilizer property for $r$ orbits of $H_0^G = \bigcup_{i=1}^r H_i^K$,
we get $|H_i^K| = \frac{|K|}{|K_{H_i}|}$ $(1 \leq i \leq r)$.
Since $\{ \pm(I_n, I_n) \}$ is a subgroup of $K_{H_i}$, we have
\[\frac{|K|}{|K_{H_i}|} \leq \frac{|K|}{2}.\]

This and Claim $1$ imply
\[\frac{(n-1)! (n-1)!}{|\mathrm{Aut}_{x_0}(H_0)|} \frac{|\mathrm{D}|^2}{2} \leq |H_0^{G_{x_0}}| \leq |H_0^G|=|\bigcup_{i=1}^r H_i^K| \leq \frac{|K|}{2}r.\]
By Claim $2$, we have
\[\frac{(n-1)! (n-1)!}{|\mathrm{Aut}_{x_0} (H_0)| |\mathrm{Aut}(X,S)| |\mathrm{Iso}(X,S)|} \leq r.\]
This completes the proof of Theorem \ref{thm:main3}
\qed

\section{Tables for isomorphism classes}\label{sec:table}

We obtain Table \ref{4} and Table \ref{16} using GAP. In these tables, $AS(n,m)$ means that the association scheme of order $n$ labeled by \verb+ #No. m +  in web-site \cite{hanakimi}, $(\ast)$ means the number of similarity classes of Hadamard matrices with respect to the $(X,S)$ and $(\ast\ast)$ means the lower bound given in Theorem \ref{thm:main3}. By \cite[page 277, 1.49]{cd} and \cite{KT}, Table \ref{hm} in this paper, we deal with only one Hadamard matrix up to equivalence for the cases $n=4, 8$.

\begin{table}[h]
\begin{center}
\begin{tabular}{|c||c|c|c|c|c|c|c|c|c|c|}\hline
$n$ & 4&8&12&16&20&24&28&32&36&40 \\ \hline\hline
\verb+#+ &1&1&1&5&3&60&487&$13710027$&$>15000000$&$>366000000000$ \\ \hline
\end{tabular}
\caption{The number of equivalence classes of Hadamard matrices of order $n$}\label{hm}
\end{center}
\end{table}

\begin{flushleft}
\textbf{Isomorphism classes of association schemes induced by Hadamard matrices of order 4}
\end{flushleft}

Let $\mathcal{H}$ be the set of all Hadamard matrices of order 4. Then $|\mathcal{H}| = 768$. \\
Put $H_0:=\begin{pmatrix}
1&1&1&1 \\
1&1&-1&-1  \\
1&-1&-1&1  \\
1&-1&1&-1
\end{pmatrix}$,
$H_1:=\begin{pmatrix}
1&1&1&1 \\
1&1&-1&-1  \\
1&-1&1&-1  \\
1&-1&-1&1
\end{pmatrix}$,

$H_2:=\begin{pmatrix}
1&1&1&1 \\
1&-1&-1&1  \\
1&-1&1&-1  \\
1&1&-1&-1
\end{pmatrix}$ and
$H_3:=\begin{pmatrix}
1&1&1&1 \\
1&-1&-1&1  \\
1&1&-1&-1  \\
1&-1&1&-1
\end{pmatrix}$. \\

Then we obtain the following:
\begin{enumerate}
\item $(X,S)=AS(4,1)$ \\
  $\mathcal{H}=H_0^{K_1}$ \\
  $H_0^{K_1}= H_1^{K_1} = H_2^{K_1} = H_3^{K_1}$ has 768 elements;
\item $(X,S)=AS(4,2)$ \\
  $\mathcal{H}=H_0^{K_2}\cupdot H_2^{K_2}$ \\
  $H_0^{K_2}= H_1^{K_2}$ has 256 elements and $H_2^{K_2} = H_3^{K_2}$ has 512 elements;
\item $(X,S)=AS(4,3)$ \\
  $\mathcal{H}=H_0^{K_3}\cupdot H_1^{K_3}\cupdot H_3^{K_3}$ \\
  $H_0^{K_3}= H_2^{K_3}$ has 384 elements, $H_1^{K_3}$ has 128 elements and $H_3^{K_3}$ has 256 elements;
\item  $(X,S)=AS(4,4)$ \\
  $\mathcal{H}=H_0^{K_4}\cupdot H_2^{K_4}$ \\
  $H_0^{K_4}= H_1^{K_4}$ has 256 elements and $H_2^{K_4} = H_3^{K_4}$ has 512 elements,
\end{enumerate}
where $K_i$ is the group given in (\ref{C}) defined by $AS(4,i)$.
Table \ref{4} is calculated according to Theorem \ref{thm:main3} and note that $|\mathrm{Aut}_{x_0}(H_0)|=6$ .

\begin{table}[h]
\begin{center}
\begin{tabular}{|c||c|c|c|c|c|c|}\hline
$(X,S)$&$|\mathrm{Aut}(X,S)|$&$|\mathrm{Iso}(X,S)|$&$(\ast)$& $(\ast\ast)$&$\begin{array}{c} \rm{Number \ of} \\ \rm{non-Schurian}\end{array}$& \\\hline\hline
$AS(4,1)$&24&24&1&1 &0&$AS(16,30)$ \\\hline
$AS(4,2)$&8&8&2& 1&1& $AS(16,54-55)$\\\hline
$AS(4,3)$&4&24&3& 1 &1& $AS(16,77-79)$\\\hline
$AS(4,4)$&4&8&2& 1 &2& $AS(16,89-90)$\\\hline
\end{tabular}
\caption{Isomorphism classes of association schemes induced by Hadamard matrices of order 4}\label{4}
\end{center}
\end{table}

\begin{flushleft}
\textbf{Isomorphism classes of association schemes induced by Hadamard matrices of order 8}
\end{flushleft}
Let $H$ be the Hadamard matrix which is obtained from $PG(2,2)$. Then $|\mathrm{Aut}_{x_0}(H)|=168$ (see \cite[Theorem 4]{kantor}).

Table \ref{16} is calculated according to Theorem \ref{thm:main3}.

\begin{table}[h]
\begin{center}
\begin{tabular}{|c||c|c|c|c|c|c|}\hline
$(X,S)$ & $|\mathrm{Aut}(X,S)|$&$|\mathrm{Iso}(X,S)|$ & $(\ast)$ &$(\ast\ast)$&$\begin{array}{c} \rm{Number \ of} \\ \rm{ non-Schurian}\end{array}$& \\ \hline\hline
$AS(8,1)$ &40320&40320& 1 & 1 &0&$AS(32,53)$
 \\  \hline
$AS(8,2)$ &384&384& 17 & 1 &17&$AS(32,4031-4047)$
 \\ \hline
$AS(8,3)$ &1152&1152& 6 & 1&6&$AS(32,4117-4122) $
\\ \hline
$AS(8,4)$ &128&128& 56 & 5   &55&  $AS(32,4529-4584)$
\\ \hline
$AS(8,5)$ &48&48& 218 & 33&217& $AS(32,4646-4863) $
\\ \hline
$AS(8,6)$ &24&48& 104 & 66&104& $AS(32,5083-5186)$
\\ \hline
$AS(8,7)$ &32&192& 130 &13&129& $AS(32,5473-5602)$
 \\ \hline
$AS(8,8)$ &32&64& 143 & 37&143& $AS(32,5745-5887)$
\\ \hline
$AS(8,9)$ &64&384& 37 & 4&37& $AS(32,6068-6104) $
\\ \hline
$AS(8,10)$ &16&32& 337 & 148&337&$AS(32,6105-6441)$
\\ \hline
$AS(8,11)$ &64&128& 60 & 10&59&$AS(32,6884-6943) $
\\ \hline
$AS(8,12)$ &16&32& 247 & 148&247&$AS(32,6944-7190) $
\\ \hline
$AS(8,13)$ &16&64& 377 & 74&376&$AS(32,7785-8161) $
\\ \hline
$AS(8,14)$ &16&64& 319 & 74&319&$AS(32,8598-8916)$
\\ \hline
$AS(8,15)$ &16&64& 286 & 74&286&$AS(32,9288-9573) $
\\ \hline
$AS(8,16)$ &16&64& 179 & 74&179&$AS(32,9892-10070)$
 \\ \hline
 $C_2 \times C_2\times C_2$&8&1344& 65 &8 &64& $AS(32,11168-11232)$
 \\ \hline
$D_4$ &8&64& 441 & 148&441&$ AS(32,11305-11745) $
\\ \hline
$C_4\times C_2$&8&64& 442 &148 &442& $AS(32,12191-12632)$
 \\ \hline
$Q_8$ &8&192& 138 & 50&138&$ AS(32,13083-13220) $
\\ \hline
$C_8$ &8&32& 462 & 296&462& $AS(32,13221-13682) $
\\ \hline
\end{tabular}
\caption{Isomorphim classes of association schemes induced by Hadamard matrices of order 8}\label{16}
\end{center}
\end{table}

\begin{flushleft}
\textbf{Isomorphism classes of association schemes induced by Hadamard matrices of order $2^n$ and a cyclic group of order $2^n$}
\end{flushleft}

Let $(X,S)=\mathfrak{I}(C_{2^n})$ (for the definition of $\mathfrak{I}(C_{2^n})$, see \cite{zies} page 177).
Then $\mathrm{Aut}(X,S)=R(C_{2^n})$ and $\mathrm{Iso}(X,S)=R(C_{2^n}) \rtimes \mathrm{Aut}(C_{2^n})$, where $R(C_{2^n})$ is the group
$(\{f_a:C_{2^n}\rightarrow C_{2^n}\mid  a\in C_{2^n}\},\circ)$ and $f_a(x)=xa$. Therefore, $|\mathrm{Aut}(X,S)| =2^n$
and $| \mathrm{Iso}(X,S)|= 2^{2n-1}$.

Let $H$ be a Hadamard matrix induced by $PG(2, n-1)$, which is called a Sylvester matrix. Then $|\mathrm{Aut}_{x_0}(H)| = (2^n -2^0)(2^n-2^1)\cdots (2^n-2^{n-1})$.
Therefore, by Theorem \ref{thm:main3}, there are at least $\frac{(2^n-1)! (2^n-1)!}{2^{3n}(2^n -2^0)(2^n-2^1)\cdots (2^n-2^{n-1})}$ isomorphism classes of association schemes which are obtained by $\mathfrak{I}(C_{2^n})$ and $H$.

\section*{Acknowledgement}
The authors would like to thank anonymous referees for their careful reading and valuable comments.

\bibstyle{plain}


\begin{thebibliography}{30}
\bibitem{bannai} E. Bannai, T. Ito, Algebraic Combinatorics I: Association Schemes, Benjamin/Cummings, Menlo Park, 1984.
\bibitem{bcn} A.E. Brouwer, A.M. Cohen, A. Neumaier, Distance-Regular Graphs, Springer-Verlag, Berlin, 1989.
\bibitem{cd} C.J. Colbourn, J.H. Dinitz, Handbook of Combinatorial Designs, Second Edition, CRC Press, 2006.
\bibitem{ega} Y. Egawa, Characterization of $H(n,q)$ by the parameters, J. Combin. Th. (A) 31 (1981) 108--125.
\bibitem{hanakimi} A. Hanaki, I. Miyamoto, Classification of association schemes of small order, Online catalogue. http://kissme.shinshu-u.ac.jp/as.
\bibitem{kantor} W.M. Kantor, Automorphism groups of Hadamard matrices, J. Combin. Th. (A) 6 (1969) 279--281.
\bibitem{KT} H. Kharaghani and B. Tayfeh-Rezaie, Hadamard matrices of order 32, J. Combin. Des. 21 (2013) 212--221.
\bibitem{mp1} M. Muzychuk, I. Ponomarenko, On pseudocyclic association schemes, Ars Math. Contemp. 5 (2012) 1--25.
\bibitem{rotman} J.J. Rotman, An Introduction to the Theory of Groups, Forth Edition, Springer, 1999.
\bibitem{ter} P. Terwilliger, The Johnson graph $J(d,r)$ is unique if $(d,r) \neq (2,8)$, Discrete Math. 58 (1986) 175--189.
\bibitem{zies} P.-H. Zieschang, An Algebraic Approach to Association Schemes, Lecture Notes in Mathematics 1628, Springer, Berlin, 1996.
\bibitem{zies2} P.-H. Zieschang, Theory of Association Schemes, Springer Monographs in Mathematics, Springer, Berlin, 2005.
\end{thebibliography}
\end{document}